\documentclass[12pt,reqno]{amsart}
\usepackage{amsmath}
\usepackage{amsthm}
\usepackage{amsfonts}
\usepackage{amssymb}
\usepackage{mathrsfs}
\usepackage[margin=1.35in, top =1.25in, bottom = 1.25in]{geometry}
\usepackage{setspace}
\usepackage{tikz}
\usetikzlibrary{matrix, arrows}
\usepackage{tikz-cd}
\usepackage{enumitem}
\usepackage[colorlinks, breaklinks]{hyperref}
\hypersetup{linkcolor=black,citecolor=blue,filecolor=black,urlcolor=black} 
\usepackage{pdflscape}
\usepackage{array}
\usepackage{adjustbox}
\usepackage{tabularx}
\usepackage{multirow}
\usepackage{multicol}
\usepackage{float}

\usepackage{color}

\DeclareMathOperator{\ass}{Ass}
\DeclareMathOperator{\Ass}{Ass}

\renewcommand{\epsilon}{\varepsilon}

\newcommand{\exterior}[2]{\bigwedge^{#1} #2}	% Exterior powers.
\DeclareMathOperator{\gr}{gr}
\DeclareMathOperator{\hgt}{ht}  			% Height of an ideal.
\DeclareMathOperator{\init}{in}				% Initial ideal of an ideal.
\newcommand{\iso}{\cong}
\newcommand{\kk}{k}					% The ground field.
	
\newcommand{\longto}{\longrightarrow}
\newcommand{\m}{\mathfrak{m}}
\DeclareMathOperator{\pd}{pd}  			% Projective dimension.
\renewcommand{\phi}{\varphi}

\newcommand{\QQ}{\mathbb{Q}}

\DeclareMathOperator{\reg}{reg}			% Castelnuovo-Mumford regularity.

\DeclareMathOperator{\syz}{Syz} 			
\newcommand{\tensor}{\otimes}
\DeclareMathOperator{\Tor}{Tor}
\DeclareMathOperator{\unmixed}{unm}

\newcommand{\graph}[2]{
	\begin{tikzpicture}          
	\newcommand*\points{#1}     
	\newcommand*\edges{#2}          
	\newcommand*\scale{0.015}          
	\foreach \x/\y/\z/\w in \points {
		\draw[fill = black!50] (\scale*\x,\scale*\y) circle [radius = 0.1] node[label = {[label distance = 0.05 cm]\w: \z}] (\z) {}; 
	}
	\foreach \x/\y in \edges { \draw (\x) -- (\y); }      
	\end{tikzpicture}
}

\newtheorem{thm}{Theorem}[section]
\newtheorem{lemma}[thm]{Lemma}
\newtheorem{prop}[thm]{Proposition}
\newtheorem{cor}[thm]{Corollary}

\newtheorem*{main-thm}{Main Theorem}

\theoremstyle{definition}

\newtheorem{rmk}[thm]{Remark}
\newtheorem{question}[thm]{Question}

\newtheorem*{notation}{Notation}

\numberwithin{equation}{section}
\numberwithin{table}{section}

\title{Betti Numbers of Koszul Algebras Defined by Four Quadrics}
\date{}

\author{Paolo Mantero}
\address{Department of Mathematics, University of Arkansas, Fayetteville, AR 72701}
\email{pmantero@uark.edu}

\author{Matthew Mastroeni}
\address{Department of Mathematics, Oklahoma State University, Stillwater, OK 74078}
\email{mmastro@okstate.edu}


\begin{document}
\maketitle

\begin{abstract}
Let $I$ be an ideal generated by quadrics in a standard graded polynomial ring $S$ over a field.  A question of Avramov, Conca, and Iyengar asks whether the Betti numbers of $R = S/I$ over $S$ can be bounded above by binomial coefficients on the minimal number of generators of $I$ if $R$ is Koszul.  This question has been answered affirmatively for Koszul algebras defined by three quadrics and Koszul almost complete intersections with any number of generators.  We give a strong affirmative answer to the above question in the case of four quadrics by completely determining the Betti tables of height two ideals of four quadrics defining Koszul algebras. 
\end{abstract}

\begin{spacing}{1.1}
\section{Introduction}

Let $\kk$ be a field, $S$ be a standard graded polynomial ring over $\kk$, $I \subseteq S$ be a graded ideal, and $R = S/I$.  We say that $R$ is a \emph{Koszul algebra}\index{Koszul!algebra} if $\kk \cong R/\m$ has a linear free resolution over $R$ where $\m = \bigoplus_{i > 0} R_i$.  Koszul algebras possess extraordinary homological properties; see the surveys \cite{Fröberg:Koszul:algebras:survey} and \cite{Koszul:algebras:and:their:syzygies}.  Moreover, examples of Koszul algebras appear throughout commutative algebra and algebraic geometry, including the coordinate rings of Grassmannians \cite{Kempf} and most canonical curves \cite{canonical:rings:of:curves}, many types of toric rings \cite{Hibi:rings} \cite{bipartite:edge:rings} \cite{strongly:Koszul:algebras}, as well as all suitably high Veronese subrings of any standard graded algebra \cite{high:Veronese:subrings:are:Koszul}.  However, the simplest examples of Koszul algebras are quotients by quadratic monomial ideals \cite{quadratic:monomial:ideals:are:Koszul}, and a recent guiding principle in the study of Koszul algebras has been that any reasonable property of algebras defined by quadratic monomial ideals should also hold for Koszul algebras; for example, see \cite{free:resolutions:over:Koszul:algebras}, \cite{Koszul:algebras:and:their:syzygies}, \cite{subadditivity:of:Betti:numbers}.  Among such properties, considering the Taylor resolution for an algebra defined by a quadratic monomial ideal leads to the following question about the Betti numbers of a Koszul algebra.

\begin{question}[{\cite[6.5]{free:resolutions:over:Koszul:algebras}}] \label{Betti:number:bound:for:Koszul:algebras}
If $R$ is Koszul and $I$ is minimally generated by $g$ quadrics, does the following inequality hold for all $i$?
\[ \beta^S_i(R) \leq \binom{g}{i} \]
In particular, is $\pd_S R \leq g$?
\end{question}

The above questions are known to have affirmative answers when $R$ is LG-quadratic (see next section).  Although LG-quadratic algebras form a proper subclass of all Koszul algebras, they encompass almost all known examples.  For general Koszul algebras, much less is known about the above question.  An affirmative answer was first given for Koszul algebras defined by $g \leq 3$ quadrics in \cite[4.5]{Koszul:algebras:defined:by:3:quadrics}.  This work was subsequently extended to Koszul almost complete intersections (where $\hgt I = g -1$) with any number of generators in \cite{Koszul:ACI's}.  Building on the successes of the preceding two papers, we show that Question \ref{Betti:number:bound:for:Koszul:algebras} has an affirmative answer when $R = S/I$ is a Koszul algebra defined by four quadrics.  

It is clear that the Betti number bound holds when $\hgt I = 4$ and when $\hgt I = 3$ by \cite{Koszul:ACI's}.  When $\hgt I = 1$, there is a linear form $z$ and a complete intersection of linear forms $J$ such that $I = zJ$ so that the minimal free resolution of $R$ is just Koszul complex on $J$, except that the first differential is multiplied by $z$. Hence, the Betti number bound also holds when $\hgt I = 1$, and it suffices to consider the case of height two ideals.  We prove:

\begin{main-thm}
Let $R = S/I$ be a Koszul algebra defined by four quadrics with $\hgt I = 2$.  Then the Betti table of $R$ over $S$ is one of the following:
\vspace{1 ex}
\begin{center}
\begin{minipage}{\textwidth}
\begin{multicols}{2}
\begin{enumerate}[label = \textnormal{(\roman*)}]
\item 
\textnormal{
\begin{tabular}{c|cccccccc}
  & 0 & 1 & 2 & 3  \\ 
\hline 
0 & 1 & -- & --  & --\\ 
1 & -- & 4 & 4  & 1 
\end{tabular}
}

\vspace{1 ex}

\item
\textnormal{
\begin{tabular}{c|cccccccc}
  & 0 & 1 & 2 & 3 & 4 \\ 
\hline 
0 & 1 & -- & --  & -- & -- \\ 
1 & -- & 4 & 3  & 1  & -- \\
2 & -- & -- & 3 & 3 & 1
\end{tabular}
}

\item
\textnormal{
\begin{tabular}{c|cccccccc}
  & 0 & 1 & 2 & 3  \\ 
\hline 
0 & 1 & -- & --  & -- \\ 
1 & -- & 4 & 3  & -- \\
2 & -- & -- & 1 & 1 
\end{tabular}
}

\vspace{1 ex}

\item
\textnormal{
\begin{tabular}{c|cccccccc}
  & 0 & 1 & 2 & 3 & 4 \\ 
\hline 
0 & 1 & -- & --  & -- & -- \\ 
1 & -- & 4 & 2  & --  & -- \\
2 & -- & -- & 4 & 4 & 1
\end{tabular}
}
\end{enumerate}
\end{multicols}
\end{minipage}
\end{center}
In particular, we have $\beta_i^S(R) \leq \binom{4}{i}$ for all $i$.
\end{main-thm}

In \cite{projective:dimension:of:4:quadrics}, Huneke-Mantero-McCullough-Seceleanu show that $\pd_S R \leq 6$ whenever $R$ is defined by an ideal generated by four quadrics.  Although their bound is sharp for general ideals of quadrics, a key point in the proof of our Main Theorem is that we can improve this bound to $\pd_S R \leq 4$ when $R$ is Koszul.  (In fact, Theorem \ref{projective:dimension:bound} shows this bound holds for a slightly larger class of algebras.)  With this result, we can then write down all possible Betti tables for $R$.  

\begin{notation} 
Throughout the remainder of the paper, the following notation will be in force unless specifically stated otherwise.  Let $\kk$ be a fixed algebraically closed ground field of arbitrary characteristic, $S$ be a standard graded polynomial ring over $\kk$,  $I \subseteq S$ be a proper graded ideal, and $R = S/I$.  Because Betti numbers are preserved under flat base extension, there is no loss of generality in assuming that the ground field is algebraically closed.  Recall that the ideal $I$ is called \emph{nondegenerate}\index{nondegenerate} if it does not contain any linear forms.  We can always reduce to a presentation for $R$ with $I$ nondegenerate by quotienting out a basis for the linear forms contained in $I$, and we will assume that this is the case throughout.  We denote the irrelevant ideal of $R$ by $\m = \bigoplus_{i > 0} R_i$.
\end{notation}

The division of the rest of the paper is as follows. In \S \ref{basic:properties:and:examples} after reviewing the relevant properties of Koszul algebras and their Betti tables, we show that every height two ideal generated by at least four independent quadrics has multiplicity at most two. In \S \ref{multiplicity:2:case:section} we refine arguments of Engheta on the unmixed parts of certain ideals to give a structure theorem for ideals generated by at least four independent quadrics over an algebraically closed field having height two and multiplicity two.  As a consequence, we see that every such ideal defines a Koszul algebra, and we can concentrate on the multiplicity one case to establish the projective dimension bound predicted by Question \ref{Betti:number:bound:for:Koszul:algebras}.  In \S \ref{projective:dimension:and:multiplicity:bounds:section}, we prove $\pd_S R \leq 4$ for Koszul algebras defined by height two ideals of four quadrics. From this, we deduce the possible Betti tables of such rings.  

Our ultimate goal is to prove a structure theorem for the defining ideals of Koszul algebras defined by four quadrics over an algebraically closed field.  We obtain results in this direction in \S \ref{multiplicity:2:case:section} when the multiplicity is two.  However, the case of height two ideals of multiplicity one is substantially more complex, and so, we have relegated that case to a separate paper \cite{structure:of:4:quadric:Koszul:algebras} to keep the current one at a manageable length.

\section{Background}
\label{basic:properties:and:examples}

\subsection{Koszul Algebras}

If $R$ is a Koszul algebra, it is well-known that its defining ideal $I$ must be generated by quadrics, but not every ideal generated by quadrics defines a Koszul algebra.  We have already noted in the introduction that every quadratic monomial ideal defines a Koszul algebra.  More generally, we say that $R$ or $I$ is \emph{G-quadratic}\index{G-quadratic} if, after a suitable linear change of coordinates $\phi: S \to S$, the ideal $\phi(I)$ has a Gr\"obner basis consisting of quadrics.  We also say that $R$ or $I$ is \emph{LG-quadratic} if $R$ is a quotient of a G-quadratic algebra $A$ by an $A$-sequence of linear forms.  Every G-quadratic algebra is Koszul by upper semicontinuity of the Betti numbers; see \cite[3.13]{upper:semicontinuity}.  It then follows from Proposition \ref{passing:Koszulness:to:and:from:quotients} below that every LG-quadratic algebra is also Koszul. In particular, every complete intersection generated by quadrics is LG-quadratic.  In summary, we have the following implications 
\begin{center}
\begin{tikzcd}
\text{G-quadratic} \rar[Rightarrow] & \text{LG-quadratic} \rar[Rightarrow] & \text{Koszul} \\
& \text{Quadratic CI} \uar[Rightarrow]
\end{tikzcd}
\end{center}
each of which is strict.  See \cite{Koszul:algebras:and:their:syzygies} for a more detailed discussion.

We will be specifically interested in the \emph{graded Betti numbers} of a Koszul algebra $R$, which are defined by $\beta_{i,j}^S(R) = \dim_\kk \Tor_i^S(R, \kk)_j$ and related to the usual Betti numbers by $\beta_i^S(R) = \sum_j \beta_{i,j}^S(R)$.  This information is usually organized into a table, called the Betti table of $R$, where the entry in column $i$ and row $j$ is $\beta_{i,i+j}^S(R)$ and zero entries are represented by ``$-$'' for readability.  See Table \ref{height:2:4-generated:edge:ideals} for examples.

Since Question \ref{Betti:number:bound:for:Koszul:algebras} is motivated by the case of quadratic monomial ideals, a natural starting point is to examine the Betti tables of such ideals.  By the well-known procedure of polarization, studying the Betti tables of all quadratic monomial ideals is equivalent to studying those of square-free quadratic monomial ideals, and such ideals can be studied combinatorially as they are in one-to-one correspondence with simple graphs.  

Given a simple graph $G$ with vertex set $ \{v_1, \dots, v_n\}$, the \emph{edge ideal} of $G$ is the square-free monomial ideal in $S = \kk[x_1, \dots, x_n]$ given by 
\[ I(G) = (x_ix_j \mid \{v_i, v_j\} \in E(G)). \]  
In particular, an edge ideal with $g$ generators corresponds to a graph with $g$ edges and has height two if and only if the minimal size of a vertex cover of $G$ is two.  The Betti tables of height two edge ideals are shown in Table \ref{height:2:4-generated:edge:ideals} below.  We will show that these are precisely the Betti tables of all Koszul algebras defined by height two ideals generated by four quadrics.  However, it should be noted that the analogy between general Koszul algebras and edge ideals is not perfect; see \cite[1.20]{Koszul:algebras:and:their:syzygies} for an example of a Koszul algebra defined by five quadrics whose $h$-polynomial cannot be realized by any monomial ideal.

\begin{table}[h!]
\begin{center}  
\begin{tabular}{|c|c|c|}
\hline
Case & $\beta^S(R)$ & Graphs
\\
\hline
(i)

&
\begin{minipage}{0.2\textwidth}
\vspace{0.2 cm}  \hspace{0.2 cm}
\begin{tabular}{c|cccccccc}
  & 0 & 1 & 2 & 3  \\ 
\hline 
0 & 1 & -- & --  & --\\ 
1 & -- & 4 & 4  & 1 
\end{tabular}
\hspace{0.2 cm}
\vspace{0.4 cm}
\end{minipage}

&
\scalebox{0.8}{
\begin{minipage}{0.25\textwidth}
\vspace{0.2 cm}
\graph{-40/40/1/left,-40/-40/2/left,40/-40/3/right,40/40/4/right}{1/2,2/3,3/4,4/1}
\vspace{0.4 cm}
\end{minipage}

\hspace{-0.75cm}
\begin{minipage}{0.25\textwidth}
\vspace{0.2 cm}
\graph{-40/40/1/left,-40/-40/2/left,20/0/3/above,80/0/4/above}{1/2,2/3,3/4,3/1}
\vspace{0.4 cm}
\end{minipage}

\hspace{-0.75cm}
\begin{minipage}{0.25\textwidth}
\vspace{0.3 cm}  
\graph{-40/40/1/left,-40/-40/2/left,20/0/3/above,80/0/4/above,140/0/5/above}{2/3,3/4,3/1,4/5}
\vspace{0.4 cm}
\end{minipage}
}

\\
\hline
(ii)

&
\begin{minipage}{0.3\textwidth}
\vspace{0.2 cm}  \hspace{0.4 cm}
\begin{tabular}{c|cccccccc}
  & 0 & 1 & 2 & 3 & 4 \\ 
\hline 
0 & 1 & -- & --  & -- & -- \\ 
1 & -- & 4 & 3  & 1  & -- \\
2 & -- & -- & 3 & 3 & 1
\end{tabular}
\hspace{0.2 cm}
\vspace{0.4 cm}
\end{minipage}

&
\scalebox{0.8}{
\begin{minipage}{0.3\textwidth}
\vspace{0.3 cm}  \hspace{0.1 cm}
\graph{-40/40/1/left,-40/-40/2/left,20/0/3/above,80/0/4/above,140/40/5/right,140/-40/6/right}{2/3,3/4,3/1,6/5}
\vspace{0.4 cm}
\end{minipage}
}

\\
\hline
(iii)

&
\begin{minipage}{0.2\textwidth}
\vspace{0.2 cm}  \hspace{0.2 cm}
\begin{tabular}{c|cccccccc}
  & 0 & 1 & 2 & 3  \\ 
\hline 
0 & 1 & -- & --  & -- \\ 
1 & -- & 4 & 3  & -- \\
2 & -- & -- & 1 & 1 
\end{tabular}
\hspace{0.2 cm}
\vspace{0.4 cm}
\end{minipage}

&
\scalebox{0.8}{
\begin{minipage}{0.3\textwidth}
\vspace{0.3 cm}  \hspace{0.1 cm}
\graph{-120/0/1/above,-60/0/2/above,0/0/3/above,60/0/4/above,120/0/5/above}{1/2,2/3,3/4,4/5}
\vspace{0.4 cm}
\end{minipage}
}

\\
\hline
(iv)

&
\begin{minipage}{0.3\textwidth}
\vspace{0.2 cm}  \hspace{0.2 cm}
\begin{tabular}{c|cccccccc}
  & 0 & 1 & 2 & 3 & 4 \\ 
\hline 
0 & 1 & -- & --  & -- & -- \\ 
1 & -- & 4 & 2  & --  & -- \\
2 & -- & -- & 4 & 4 & 1
\end{tabular}
\hspace{0.2 cm}
\vspace{0.4 cm}
\end{minipage}

&
\scalebox{0.8}{
\begin{minipage}{0.25\textwidth}
\vspace{0.2 cm}
\graph{-60/40/1/above,0/40/2/above,60/40/3/above,-60/-40/4/below,0/-40/5/below,60/-40/6/below}{1/2,2/3,4/5,5/6}
\vspace{0.4 cm}
\end{minipage}
}

\\
\hline

\end{tabular}
\end{center}
   \caption{Betti tables of height 2 edge ideals with 4 generators}
   \label{height:2:4-generated:edge:ideals}
\end{table}

It is also useful to know how the Koszul property can be passed to and from quotient rings.

\begin{prop}[{\cite[\S 3.1, 2]{Koszul:algebras:and:regularity}}] \label{passing:Koszulness:to:and:from:quotients}
Let $S$ be a standard graded $\kk$-algebra and $R$ be a quotient ring of $S$.
\begin{enumerate}[label = \textnormal{(\alph*)}]
\item If $S$ is Koszul and $\reg_S(R) \leq 1$, then $R$ is Koszul.
\item If $R$ is Koszul and $\reg_S(R)$ is finite, then $S$ is Koszul.
\end{enumerate}
\end{prop}

Here, the \emph{regularity} of $R$ over $S$ is defined by 
\begin{equation} \label{regularity}
\reg_S(R) = \max\{j \mid \beta_{i,i+j}^S(R) \neq 0 \; \text{for some}\; i \}.
\end{equation}
When $S$ is a standard graded polynomial ring, this definition agrees with the usual  Castelnuovo-Mumford regularity of $R$ and does not depend on the polynomial ring $S$, so we drop the subscript from the notation.  We will primarily use the above proposition to show that the Koszul property can be passed to or from a ring and its quotient by a regular sequence consisting of linear forms or quadrics.

\subsection{Betti Table Restrictions}

There are several restrictions on the shape of the Betti tables of Koszul algebras that make it possible to deduce the Betti tables of Koszul algebras defined by four quadrics.  The first of these restrictions, discovered in \cite{Backelin} and \cite[4]{Kempf}, says that the Betti tables of Koszul algebras have nonzero entries only on or above the diagonal; see \cite[2.10]{Koszul:algebras:and:their:syzygies} for an easier argument using regularity.

\begin{lemma} \label{Koszul:algebras:have:subdiagonal:Betti:table}
If $R = S/I$ is a Koszul algebra, then $\beta_{i,j}^S(R) = 0$ for all $i$ and $j > 2i$.
\end{lemma}

Additionally, the extremal portions of the Betti table of a Koszul algebra $R$, namely the diagonal entries and the linear strand of $I$, satisfy bounds similar to those in Question \ref{Betti:number:bound:for:Koszul:algebras}.

\begin{prop}[{\cite[3.4, 4.2]{Koszul:algebras:defined:by:3:quadrics}\cite[3.2]{free:resolutions:over:Koszul:algebras}}] \label{Koszul:Betti:table:constraints}
Suppose that $R = S/I$ is Koszul and that $I$ is minimally generated by $g$ elements.  Then:
\begin{enumerate}[label = \textnormal{(\alph*)}]
\item $\beta_{i,i+1}^S(R) \leq \binom{g}{i}$ for $2 \leq i \leq g$, and if equality holds for $i = 2$, then $I$ has height one and a linear resolution of length $g$.
\item $\beta_{i,2i}^S(R) \leq \binom{g}{i}$ for $2 \leq i \leq g$, and if equality holds for some $i$, then $I$ is a complete intersection.
\item $\beta_{i,2i}^S(R) = 0$ for $i > \hgt I$.
\end{enumerate}
\end{prop}

\begin{cor} \label{nonACI:Koszul:algebras:have:2:linear:syzygies}
If $R = S/I$ is a Koszul algebra and $\hgt I \leq g-2$, then $\beta_{2,3}^S(R) \geq 2$.
\end{cor}

\begin{proof}
First, note that $\beta_{2,3}^S(R) \geq 1$.  Indeed, if this is not the case and $I$ is minimally generated by quadrics $q_1, \dots, q_g$, then the Koszul syzygies on the $q_i$ are all minimal generators of $\syz_1^S(I)$ so that $\beta_{2,4}^S(R) \geq \binom{g}{2}$, contradicting the preceding proposition since $I$ is not a complete intersection.  If $\beta_{2,3}^S(R) = 1$, it follows from \cite[4.1]{Koszul:ACI's} that $\hgt I = g -1$, which is also a contradiction.  Thus, we have $\beta_{2,3}^S(R) \geq 2$.
\end{proof}

\begin{prop}[\cite{high:linear:syzygies}] \label{high:linear:syzygy:implies:nonzero:socle}
Let $M$ be a finitely generated graded module over $S = \kk[x_1, \dots, x_r]$.  If $\beta^S_{r,r+j}(M) \neq 0$, then $(0:_M S_+)_j \neq 0$.
\end{prop}

\begin{proof}
If $K_\bullet$ denotes the Koszul complex on the variables of $S$, then $H_r(K_\bullet \tensor_S M)_{r+j} = \Tor_r^S(\kk, M)_{r+j} \neq 0$ by assumption so that there is a nonzero cycle $e_{[r]} \tensor y \in \exterior{r}{S(-1)^r} \tensor_S M$ of degree $r + j$.  Here, $e_1, \dots, e_r$ denotes the standard basis of $S(-1)^r$, and for each subset $J = \{j_1, \dots, j_p\} \subseteq [r] = \{1, \dots, r\}$ with $j_1 < \cdots < j_p$, we set $e_J = e_{j_1} \wedge \cdots \wedge e_{j_p}$.  In particular, we note that $y$ must have degree $j$.  Since $e_{[r]} \tensor y$ is a cycle, we have $\sum_{i = 1}^r (-1)^{i+1} e_{[r] \setminus \{i\}} \tensor x_i y = 0$ so that $x_iy = 0$ in $M$ for all $i$, and hence $y \in (0:_M S_+)_j$.
\end{proof}

\begin{cor} \label{projective:dimension:level:linear:syzygy}
Let $R = S/I$ where $I$ is minimally generated by $g$ quadrics and $\pd_S R = r$.  If $\beta_{r,r+1}^S(R) \neq 0$, then $r \leq g$ with equality if and only if $\hgt I = 1$.
\end{cor}

\begin{proof}
By faithfully flat base change, we may assume that the ground field $\kk$ is infinite.  Then after replacing $R$ and $S$ with their quotients by a maximal regular sequence of linear forms on $R$, we may further assume that $\dim S = \pd_S R = r$ so that $S = \kk[x_1, \dots, x_r]$.  In particular, we note that this does not affect the height one condition since, for example, the height of $I$ is determined by the Betti numbers of $R$, which are unaffected by killing a regular sequence of linear forms.  In this setup, it follows from the preceding proposition that there is a linear form $\ell$ such that $x_i\ell \in I$ for all $i$.  Moreover, these elements are linearly independent quadrics, so they must be part of a minimal set of generators for $I$.  Hence, we have $r \leq g$, and furthermore, equality holds if and only if $I = \ell S_+$, which is equivalent to $\hgt I = 1$.
\end{proof}

\begin{prop}[{\cite[3.1]{minimal:homogeneous:linkage}}]
Suppose that $R = S/I$ is a quadratic Cohen-Macaulay ring. Then $\reg R \leq \pd_S R$, and equality holds if and only if $R$ is a complete intersection.
\end{prop}

Combining the above with preceding results, we have a similar statement for Koszul algebras; the first part was proved in a more general form in \cite[3.2]{free:resolutions:over:Koszul:algebras}.

\begin{cor} \label{regularity:bounded:by:projective:dimension}
If $R = S/I$ is Koszul, then $\reg R \leq \pd_S R$, and equality holds if and only if $R$ is a complete intersection.
\end{cor}

\begin{proof}
If $\reg R = j$, then $\beta_{i,i+j}^S(R) \neq 0$ for some $i$, and so, we must have $\reg R = j \leq i \leq \pd_S R$ by Lemma \ref{Koszul:algebras:have:subdiagonal:Betti:table}.  Furthermore, if $\reg R = \pd_S R$, we see that $i = j = \pd_S R$ so that $\beta_{i,2i}^S(R) \neq 0$.  It then follows from part (c) of Proposition \ref{Koszul:Betti:table:constraints} that $\hgt I = \pd_S R$ so that $R$ is Cohen-Macaulay.  The rest of the statement then follows from the preceding proposition.
\end{proof}

\subsection{Linkage and Multiplicity}

We recall a special case of a classical linkage result that we will employ several times.

\begin{thm}[{\cite[3.1(a)]{Poincaré:series:over:rings:of:small:embedding:codepth}}] \label{AKM}
An ideal $I \subseteq S$ is directly linked to a complete intersection $J = (f,g)$ via a complete intersection $C$ if and only if $I=I_2(M)$ where  
\[ M = \begin{pmatrix} g & a_1 & a_2 \\ -f & b_1 & b_2 \end{pmatrix} \]
and  $a_1, a_2, b_1, b_2 \in S$ such that $C = (a_1f+b_1g, a_2f+b_2g)$.
\end{thm}

Recall that a \emph{linear prime} in $S$ is a prime ideal generated by linear forms.  In the next theorem, we need the following result attributed to P.~Samuel.  This result follows easily from the local case \cite[40.6]{Nagata} after replacing $R$ with $R_{\m}$ since $R_\m$ being regular forces $\gr_{\m R_\m}(R_\m) \iso R$ to be a polynomial ring \cite[1.1.8]{Bruns:Herzog}.

\begin{prop}\label{linear:primes}
If $I \subseteq S$ is an unmixed ideal such that $R = S/I$ has $e(R) = 1$, then $I$ is a linear prime.
\end{prop}

We conclude this section by proving that for any ideal $I$ of height 2 generated by $g\geq 4$ quadrics, one may only have $e(R)\leq 2$.  Recall that the \emph{unmixed part} $I^{\unmixed}$ of $I$ is the intersection of all primary components $J$ of $I$ with $\hgt J = \hgt I$.

\begin{thm} \label{multiplicity:at:most:2}
Suppose that $R = S/I$ where $\hgt I = 2$ and $I$ is generated by quadrics.  Then $e(R) \leq 3$ with equality if and only if $I^{\unmixed} = I_2(M)$ for some $3 \times 2$ matrix $M$ of linear forms. Hence, $e(R) \leq 2$ if $I$ is minimally generated by $g \geq 4$ quadrics.
\end{thm}

\begin{proof} 
Recall that the ground field $\kk$ is algebraically closed by assumption.  We already know by \cite[3.1]{projective:dimension:of:height:2:ideals:of:quadrics} that $e(R) \leq 3$. Assume that $e(R)=3$.  Let $C$ be a complete intersection of quadrics in $I$, and set $J = (C : I)$. By linkage, $J$ is a homogeneous unmixed ideal directly linked to $I^{\unmixed}$ with $e(S/J)=e(S/C)-e(R) = 1$ (see for example \cite[21.23]{Eisenbud}). It follows by  Proposition \ref{linear:primes} that $J$ is a linear prime of height 2, and thus, $J$ is a complete intersection. Consequently, Theorem \ref{AKM} implies that $I^{\unmixed} = I_2(M)$ for some $3 \times 2$ matrix $M$ of linear forms.  Conversely, if $I^{\unmixed} = I_2(M)$ for some $3 \times 2$ matrix of linear forms, then $\hgt I_2(M) = 2$ so that it is easily computed that $e(R) = e(S/I^{\unmixed}) = 3$ from the Hilbert-Burch resolution \cite[1.4.17]{Bruns:Herzog}.  In particular, $I_2(M)$ is an ideal generated by 3 quadrics, so it is impossible for $I$ to contain four linearly independent quadrics. Thus, we must have $e(R)\leq 2$ if $I$ is minimally generated by $g \geq 4$ quadrics. 
\end{proof}

\section{The Multiplicity 2 Case}
\label{multiplicity:2:case:section}

The case of Koszul algebras defined by height two ideals $I$ of multiplicity two generated by four quadrics is particularly simple to analyze thanks to a result of Engheta on the unmixed part of such ideals. 

\begin{thm} \label{multiplicity:2:case}
Let $R = S/I$ be a ring defined by $g \geq 4$ quadrics.  Then $\hgt I = e(R) = 2$ if and only if $I$ has one of the following forms:
\begin{enumerate}[label = \textnormal{(\roman*)}]
\item[\textnormal{(i$_\text{a}$)}] $(x, y) \cap (z, w)$ or $\left((x, y)^2, xz + yw\right)$ for independent linear forms $x$, $y$, $z$ and $w$, in which case we must have $g = 4$.
\item[\textnormal{(i$_\text{b}$)}] $(a_1x, \dots, a_{g-1}x, q)$ for independent linear forms $a_1, \dots, a_{g-1}$ and some linear form $x$ and quadric $q \in (a_1, \dots, a_{g-1}) \setminus (x)$.
\stepcounter{enumi}
\item $(a_1x, \dots, a_{g-1}x, q)$ for independent linear forms $a_1, \dots, a_{g-1}$ and some linear form $x$ and quadric $q$ which is a nonzerodivisor modulo $(a_1x, \dots, a_{g-1}x)$.
\end{enumerate}
\end{thm}

\begin{proof}
If $I$ has one of the forms listed above, it easily follows that $\hgt I = e(R) = 2$ by considering the minimal primes of the given ideals or by computing their Hilbert series from the Betti tables in Corolllary \ref{multiplicity:2:case:Betti:tables} below, so we prove only that every ideal $I$ with $\hgt I = e(R) = 2$ has one of these forms.   Suppose that $\hgt I = e(R) = 2$.  Since $I^{\unmixed}$ is a height two ideal with $e(S/I^{\unmixed}) = e(R) = 2$, \cite[Prop 11]{Engheta} yields that either: (1) $I^{\unmixed} = (x, y) \cap (z, w) = (xz, xw, yz, yw)$ for some independent linear forms $x$, $y$, $z$, and $w$, (2) $I^{\unmixed} = \left((x, y)^2,xz+yw\right)$ for some linear forms $x$ and $y$ and some forms $z$ and $w$ of degree $d > 0$ such that $(x, y, z, w)$ is a complete intersection, or (3) $I^{\unmixed} = (x, q)$ for some linear form $x$ and quadric $q$.

\textsc{Case (1):}  If $I^{\unmixed} = (xz, xw, yz, yw)$ for independent linear forms $x$, $y$, $z$, and $w$, then since $I \subseteq I^{\unmixed}$ contains four independent quadrics, we must have $g = 4$ and $I = I^{\unmixed}$.

\textsc{Case (2):}  If $I^{\unmixed} = \left((x, y)^2,xz+yw\right)$ for some linear forms $x$ and $y$ and some forms $z$ and $w$ of degree $d > 0$ such that $(x, y, z, w)$ is a complete intersection, then since $I \subseteq I^{\unmixed}$ is generated by four independent quadrics, we see that $z$ and $w$ must be linear forms and that $I = I^{\unmixed}$.

\textsc{Case (3):}  Assume that $I^{\unmixed} = (x, q)$ for some linear form $x$ and quadric $q$.  Then $I = (a_1x + c_1q, \dots, a_gx + c_gq)$ for some $c_i \in \kk$ and linear forms $a_i$.  If $c_i = 0$ for all $i$, then $\hgt I = 1$, which is a contradiction.  So, we may assume without loss of generality that $c_g = 1$ and, after replacing $q$ with $a_gx + q$ and subtracting combinations of the generators of $I$, that $I$ has the form $I = (a_1x, \dots, a_{g-1}x, q)$.  In particular, $a_1, \dots, a_{g-1}$ must be independent linear forms as $I$ is generated by $g$ independent quadrics.

Next, we note that the only associated primes of $(a_1x, \dots, a_{g-1}x)$ are $(x)$ and $(a_1, \dots, a_{g-1})$.  This is clear by considering the primary decomposition of an ideal of this form.  Indeed, if $a_1, \dots, a_g, x$ are independent linear forms, then we have $(a_1x, \dots, a_{g-1}x) = (x) \cap (a_1, \dots, a_{g-1})$.  Otherwise, we must have $x \in (a_1, \dots, a_{g-1})$ so that, after rescaling $x$ and possibly interchanging the $a_i$, we may assume that $x = c_1a_1 + \cdots + c_{g-2}a_{g-2} + a_{g-1}$ for some $c_i \in \kk$.  Hence, $a_1, \dots, a_{g-2}, x$ are independent linear forms, and $(a_1x, \dots, a_{g-1}x) = (a_1x, \dots, a_{g-2}x, x^2) = (x) \cap (a_1, \dots, a_{g-2}, x^2)$ so that $(x)$ and $(a_1, \dots, a_{g-1}) = (a_1, \dots, a_{g-2}, x)$ are the only associated primes.  

Finally, we note that $q \notin (x)$ since otherwise we would have $\hgt I = 1$, and so, either $q$ is a nonzerodivisor modulo $(a_1x, \dots, a_{g-1}x)$ or $q \in (a_1, \dots, a_{g-1}) \setminus (x)$ by the preceding discussion.  
\end{proof}

\begin{rmk}
Although Engheta does not explicitly state the assumption in his classification of unmixed ideals of height two and multiplicity two, working over an algebraically closed ground field in the above theorem is necessary to ensure that prime ideals of height two and multiplicity two contain a linear form \cite[18.12]{Harris:algebraic:geometry} so that they fall into case (3) above.  For example, it is easily checked in Macaulay2 \cite{Macaulay2} that the ideal
\[ I_2\begin{pmatrix} x & y & z & w \\ -y & x & -w & z \end{pmatrix} \subseteq \QQ[x, y, z, w] \]
is a nondegenerate prime ideal of height two and multiplicity two.
\end{rmk}

\begin{cor} \label{multiplicity:2:case:Betti:tables}
Let $R = S/I$ where $I$ is an ideal minimally generated by $g \geq 4$ quadrics of one of the forms described in the preceding theorem.  Then the Betti table of $R$ is one of the following:
\vspace{1 ex}
\begin{center}
\textnormal{
\begin{tabular}{l}
\begin{tabular}{c|cccccccc}
  & 0 & 1 & 2 & 3 & $\cdots$ & $g-2$ & $g-1$ \\ 
\hline 
0 & 1 & -- & --  & -- & $\cdots$ & -- & -- \\ 
1 & -- & $g$ & $\binom{g-1}{2} + 1$  & $\binom{g-1}{3}$ & $\cdots$ & $\binom{g-1}{g-2}$ & 1 
\end{tabular} 
\\
\\[1 ex]
\begin{tabular}{c|ccccccccc}
  & 0 & 1 & 2 & $\cdots$  & $\cdots$ & $g-2$ & $g-1$ & $g$ \\ 
\hline 
0 & 1 & -- & --  & $\cdots$ & $\cdots$ &  -- & -- & -- \\ 
1 & -- & $g$ & $\phantom{ + }\binom{g-1}{2}\;$  & $\cdots$ &  $\cdots$ & $\binom{g-1}{g-2}$ & 1  & -- \\
2 & -- & -- & $g-1$ & $\cdots$ &  $\phantom{ + }\cdots\phantom{ + }$ & $\binom{g-1}{g-3}$ & $\binom{g-1}{g-2}$ & 1
\end{tabular}
\end{tabular}
}
\end{center}
\vspace{1 ex}
with the former corresponding to cases \textnormal{(i$_\text{a}$)} and \textnormal{(i$_\text{b}$)} and the latter to case \textnormal{(ii)}.
\end{cor}

\begin{proof}
Suppose that $I = (a_1x, \dots, a_{g-1}x, q)$ for independent linear forms $a_1, \dots, a_{g-1}$, a linear form $x$, and a quadric $q$.  Using the short exact sequence 
\[ 0 \to S/((a_1x, \dots, a_{g-1}x): q)(-2) \stackrel{q}{\to} S/(a_1x, \dots, a_{g-1}x) \to R \to 0 \]
we can easily compute the minimal free resolution of $R$ as a mapping cone.  The minimal free resolution $F_\bullet$ of $S/(a_1x, \dots, a_{g-1}x)$ is the Koszul complex on $a_1, \dots, a_{g-1}$ except that the first differential is multiplied by $x$ and the resolution must be twisted by $-1$ accordingly.  It has the Betti table:
\begin{center}
\begin{tabular}{c|cccccccc}
  & 0 & 1 & 2 & $\cdots$ & $g-2$ & $g-1$  \\ 
\hline 
0 & 1 & -- & --  & $\cdots$  & -- & -- \\ 
1 & -- & $g-1$ & $\binom{g-1}{2}$ & $\cdots$ & $\binom{g-1}{g-2}$ & 1   \\
\end{tabular}
\end{center}
If $q$ is nonzerodivisor modulo $(a_1x, \dots, a_{g-1}x)$, then taking a mapping cone yields the second Betti table above.  If $q \in (a_1, \dots, a_{g-1}) \setminus (x)$, then $((a_1x, \dots, a_{g-1}x): q) = (x)$ so that taking the mapping cone of a suitable lift of multiplication by $q$ to a chain map $K_\bullet(x)(-2) \to F_\bullet$ yields the first Betti table above.

When $I = \left((x, y)^2,xz+yw\right)$ for independent linear forms $x$, $y$, $z$, and $w$, we note that $((x, y)^2: xz+yw) = (x, y)$ so that we can again compute the minimal free resolution of $R$ as mapping cone coming from the short exact sequence $0 \to S/(x,y)(-2) \stackrel{xz+yw}{\to} S/(x,y)^2 \to R \to 0$.  In this case, the chain map lifting multiplication by $xz+yw$ on the minimal free resolutions of $S/(x,y)(-2)$ and $S/(x,y)^2$ is:
\[
\begin{tikzcd}[ampersand replacement = \&, column sep = 2 cm]
0 \rar \& S(-4) \rar{\left(\begin{smallmatrix} y \\ -x \end{smallmatrix}\right)} \dar[swap]{\left(\begin{smallmatrix} z \\ w \end{smallmatrix}\right)} \& S(-3)^2 \rar{\left(\begin{smallmatrix} x & y \end{smallmatrix}\right)} \dar{\left(\begin{smallmatrix} z & 0 \\ w & z \\ 0 & w \end{smallmatrix}\right)} \& S(-2) \dar{xz+yw}
\\
0 \rar \& S(-3)^2 \rar{\left(\begin{smallmatrix} y & 0 \\ -x & y \\ 0 & -x \end{smallmatrix}\right)} \& S(-2)^3 \rar{\left(\begin{smallmatrix} x^2 & xy & y^2 \end{smallmatrix}\right)} \& S
\end{tikzcd}
\]
Taking the mapping cone yields a minimal free resolution for $R$ realizing the first Betti table above with $g = 4$.

Finally, if $I = (x, y) \cap (z, w) = (xz, xw, yz, yw)$ for independent linear forms $x$, $y$, $z$, and $w$, we note that $((xz, xw, yz): yw) = (x, z)$ so that we can again compute the minimal free resolution of $R$ as mapping cone coming from the short exact sequence $0 \to S/(x,z)(-2) \stackrel{yw}{\to} S/(xz, xw, yz) \to R \to 0$.  Since the middle term has a Hilbert-Burch resolution, the minimal free resolution of $R$ realizes the first Betti table above in a manner similar to the previous case.  
\end{proof}

Combining the preceding theorem and corollary with Theorem \ref{multiplicity:at:most:2}, we have the following.

\begin{cor}
Question \ref{Betti:number:bound:for:Koszul:algebras} has a positive answer if $R=S/I$ when $\hgt I =2$ and $e(R)>1$.
\end{cor}

In retrospect, however, this turns out to be unsurprising as such Koszul algebras turn out to be G-quadratic up to extension of the ground field.  The proof of the next theorem assumes a familiarity with various results about Gr\"obner bases.  We refer the reader unfamiliar with these results to \cite{Herzog:Ene}, \cite{ideals:varieties:and:algorithms}, or \cite{Eisenbud} for further details and any unexplained terminology.

\begin{thm} \label{multiplicity:2:case:is:G-quadratic}
Let $R = S/I$ be a ring defined by $g \geq 4$ quadrics with $\hgt I = e(R) = 2$. Then $R$ is G-quadratic.
\end{thm}

\begin{proof}
We argue according to the possibilities for $I$ described in the preceding theorem.

\textsc{Case} (i$_\text{a}$): If $I = (x, y) \cap (z, w)$, then $I$ is a monomial ideal up to a linear change of coordinates so that $R$ is G-quadratic.  If $I = (x, y)^2+(xz+yw)$, then after a linear change of coordinates, we may assume that $x$, $y$, $z$, and $w$ are variables of $S$, and it is easily checked via Buchberger's Criterion that the generators of $I$ form a universal quadratic Gr\"obner basis.  Hence, $R$ is G-quadratic.

\textsc{Case} (i$_\text{b}$): Suppose that $I = (a_1x, \dots, a_{g-1}x, q)$ where $q \in (a_1, \dots, a_{g-1}) \setminus (x)$.  By \cite[4]{spaces:of:quadrics:of:low:codimension}, it suffices to prove that $R$ is G-quadratic after quotienting out a regular sequence of linear forms on $R$.  Hence, we may assume that $\dim S = \pd_S R = g - 1$.  In that case, since $a_1, \dots, a_{g-1}$ are independent linear forms, we may assume after a suitable linear change of coordinates that $a_1, \dots, a_{g-1}$ are the variables of $S$.  As $x$ is a nonzero linear form, it must contain one of the $a_i$ in its support.  Without loss of generality, after rescaling $x$ and performing a linear change of coordinate, we may assume that $x = a_{g-1} + \sum_{i = 1}^{g-2} \lambda_i a_i$ for some $\lambda_i \in \kk$.  After a further linear change of coordinates replacing $a_{g-1}$ with $a_{g - 1} - \sum_{i = 1}^{g-2} \lambda_i a_i$ and fixing all other variables and after suitable change of generators for $I$, we have that $I = (a_1a_{g-1}, \dots, a_{g-2}a_{g-1}, a_{g-1}^2, q)$ where $q \in \kk[a_1, \dots, a_{g-2}]$.  

Since $q \neq 0$, we may assume a linear change of coordinates that $q$ contains $a_1^2$ in its support.  Indeed, this is clear if $q$ contains $a_i^2$ for some $i$ by relabeling, so we may assume that $q$ is square-free.  In that case, as $g \geq 4$, we may assume that the support of $q$ contains the monomial $a_1a_2$ after relabeling so that a change of coordinates replacing $a_2$ with $a_1 + a_2$ and fixing all other variables of $S$ yields $q$ with containing $a_1^2$ in its support.  

Once we have that $a_1^2$ is in the support of $q$ with coefficient $1$, the generators of $I$ are a quadratic Gr\"obner basis.  To see this, we need only check Buchberger's criterion for the $g-1$ S-pairs of the monomial generators with $q$ since it is clear that the S-pairs of the monomials are zero.  If we choose any monomial order with $a_1$ bigger than all other variables in $S$, it is clear that the S-pairs of $q$ with each of $a_2a_{g-2}, \dots, a_{g-2}a_{g-1}, a_{g-1}^2$ reduce to zero since the leading monomial $a_1^2$ of $q$ is relatively prime to each of these other monomials \cite[2.15]{Herzog:Ene}.  Finally, by writing $q = \sum_{1 \leq i \leq j \leq g-2} \alpha_{i,j} a_ia_j$ for some $\alpha_{i, j} \in \kk$, we see that 
\[ 
S(a_1a_{g-1}, q) = a_1(a_1a_{g-1}) - a_{g-1}q = \sum_{j = 2}^{g-2} \left(\sum_{i = 1}^j \alpha_{i,j}a_i\right)(a_ja_{g-1}) 
\]
is a standard expression for $S(a_1a_{g-1}, q)$ with zero remainder since no cancellation can occur among the monomials in the forms $(\sum_{i=1}^j \alpha_{i,j} a_i)a_ja_{g-1}$. 

\textsc{Case} (ii):  As in the proof of the previous case, we may assume that $\dim S = \pd_S R = g$ after quotienting out by a maximal $R$-sequence of linear forms, and since $a_1, \dots, a_{g-1}$ are independent linear forms, we may further assume after a suitable linear change of coordinates that the $a_i$ are variables.  Let $a_g$ denote the remaining variable of $S$.  Since $q \notin (a_1, \dots, a_{g-1})$, we must have that $a_g^2$ is in the support of $q$.  

There are two cases to consider.  If $x$ does not contain $a_g$ in its support, then in any monomial order $<$ with $a_g$ larger than every other variable, we have that the S-pairs $S(a_ix, a_jx) = 0$ for all $i, j < g$ and $S(a_ix, q)$ reduces to zero for all $i < g$ since $\init_<(q) = a_g^2$ is relatively prime to $\init_<(a_ix) \in (a_1, \dots, a_{g-1})$ \cite[2.15]{Herzog:Ene}.  Hence, $a_1x, \dots, a_{g-1}x, q$ is a Gr\"obner basis for $I$.  Thus, we may assume for the rest of the proof that $a_g$ is in the support of $x$.

In that case, after rescaling $x$, we can write $x = a_g + \sum_{i = 1}^{g-1} \lambda_i a_i$ for some $\lambda_i \in \kk$.  After a linear change of coordinates replacing $a_g$ with $a_g - \sum_{i = 1}^{g-1} \lambda_ia_i$ and fixing all other variables, we may assume that $x = a_g$ so that $I = (a_1a_g, \dots, a_{g-1}a_g, q)$.  Note that this does not affect that $a_g^2$ must be in the support of $q$, and so, after a suitable change of generators for $I$, we have $q = a_g^2 + q'$ where $q' \in \kk[a_1, \dots, a_{g-1}]$.  Since $\hgt I = 2$, we must have $q' \neq 0$, and so, we may further assume that $a_1^2$ is in the support of $q'$ after a suitable change of coordinates.  Let $\delta \in \kk$ such that $\delta^2$ is the coefficient of $a_1^2$ in $q'$ (which exists since $\kk$ is algebraically closed).  Replacing $q$ with $q - 2\delta a_1a_g$ as a generator for $I$, we may assume $q = (a_g - \delta a_1)^2 + q'$ for some $q' \in \kk[a_1, \dots, a_{g-1}]$ not containing $a_1^2$ in its support, and after performing a change of coordinates replacing $a_g$ with $a_g + \delta a_1$ and fixing all other variables, we obtain $I = (a_1x, \dots, a_{g-1}x, q)$ where $x = a_g + \delta a_1$ and $q = a_g^2 + q'$.  If $<$ is the product order obtained from the degree lexicographic order on $\kk[a_1, a_g]$ with $a_1 > a_g$ followed by any other monomial order on $\kk[a_2, \dots, a_{g-1}]$, then $\init_<(q) = a_g^2$ since the only quadratic monomials larger than $a_g^2$ are $a_1^2$ and $a_1a_g$, which do not belong to the support of $q$.  And so, we have that the S-pairs $S(a_ix, a_jx) = 0$ for all $i, j < g$ and $S(a_ix, q)$ reduces to zero for all $i < g$ since $\init_<(q) = a_g^2$ is relatively prime to $a_1a_i$ \cite[2.15]{Herzog:Ene}.  Hence, $a_1x, \dots, a_{g-1}x, q$ is a Gr\"obner basis for $I$.
\end{proof}

\begin{rmk}
In the proof of preceding theorem, we needed the ground field to be algebraically closed in order to show that rings $R = S/I$ of the form (ii) in Theorem \ref{multiplicity:2:case} are G-quadratic.  If we relax the assumption about the ground field being algebraically closed, rings of this form are still easily seen to be LG-quadratic by first forming the G-quadratic algebra $A = S[u]/(a_1x, \dots, a_{g-1}x, u^2 + q)$ since $u$ is a nonzerodivisor on $A$ such that $A/(u) \iso R$.  This suggests a potential place to look for an affirmative answer to the following question.
\end{rmk}

\begin{question}
Does there exist a ring $R = S/I$ which is not G-quadratic but which becomes G-quadratic after a suitable extension of the ground field?
\end{question}

\section{The Projective Dimension Bound and Betti Tables}
\label{projective:dimension:and:multiplicity:bounds:section}

The goal of this section is to determine all Betti tables of Koszul algebras $R=S/I$ where $I$ is generated by $g=4$ quadrics and $\hgt I = 2$. For a height two ideal generated by four quadrics, it was proved in \cite{{projective:dimension:of:height:2:ideals:of:quadrics}} that $\pd_S R \leq 6$, which is a sharp bound in general.

When $R$ is Koszul, we improve this bound to $\pd_S R \leq 4$ (see Corollary \ref{Betti:number:4:quadrics}). When $\hgt I = 1, 4$, the bound is clear, and when $\hgt I = 3$, it was proved in \cite{{Koszul:ACI's}}, so it suffices to prove it when $\hgt I = 2$. This is achieved in Theorem \ref{projective:dimension:bound} below, where we prove the bound under the weaker assumption that $I$ has two linear syzygies.  Then we determine all possible Betti tables of $R$ (Theorem \ref{Koszul:4:quadric:height:2:Betti:tables}), from which we obtain the positive answer to Question \ref{Betti:number:bound:for:Koszul:algebras}.

\begin{thm} \label{projective:dimension:bound}
Let $I=(q_1,q_2,q_3,q_4) \subseteq S$ be an ideal of height two generated by four linearly independent quadrics.  If $I$ has two independent linear syzygies, then $\pd_S R \leq 4$. 
\end{thm}

The proof of the preceding theorem is divided into a number of lemmas. First, by the previous section, we may focus on the case $e(R)=1$. Then, by the Associativity Formula, there are linear forms $x$ and $y$ such that $P = (x, y)$ is the only minimal prime of $I$ of height 2.  Following \cite[\S 4]{projective:dimension:of:height:2:ideals:of:quadrics}, we write $q_i=a_ix+b_iy$ for some linear forms $a_i$ and $b_i$, and we set:
\begin{equation*}
M = \begin{pmatrix}
y & a_1 & a_2 & a_3 & a_4  \\
-x & b_1 & b_2 & b_3 & b_4  
\end{pmatrix} 
\qquad 
A = \begin{pmatrix}
a_1 & a_2 & a_3 & a_4  \\
b_1 & b_2 & b_3 & b_4  
\end{pmatrix}
\end{equation*}
We recall that $I$ is said to be \emph{represented by minors} by the matrix $M$ and \emph{represented by coefficients} by the matrix $A$. As observed in \cite[4.4]{projective:dimension:of:height:2:ideals:of:quadrics}, these matrices need not be unique.  Also, we remark that $I_2(M)=I+I_2(A)$.

\begin{prop}[{\cite[4.7]{projective:dimension:of:height:2:ideals:of:quadrics}}] \label{representation:by:minors}
Let $M$ be a $2\times(n+1)$ matrix of linear forms with $n\geq 2$ representing a height two ideal of quadrics by minors. Then $M$ is equivalent via a sequence of ideal-preserving elementary operations to a matrix of linear forms $M'$ of one of the following types, where $\hgt(x,y)=2$:
\begin{enumerate}[label = \textnormal{(\arabic*)}]

\item $M'$ is 1-generic.

\item $M' = 
\begin{pmatrix}
y & 0 & a_2 & \ldots & a_n \\ 
-x & b_1 & b_2 & \ldots & b_n
\end{pmatrix}$, where  $D = 
\begin{pmatrix}
y & a_2 & \ldots & a_n \\ 
-x & b_2 & \ldots & b_n
\end{pmatrix}$ is 1-generic.

\item $M' = 
\begin{pmatrix}
y & 0 & 0 & a_3 & \ldots & a_n \\ 
-x & b_1 & b_2 & b_3 & \ldots & b_n
\end{pmatrix}$, with no additional restrictions.

\item $M' = 
\begin{pmatrix}
y & 0 & a_2 & a_3 & \ldots & a_n \\ 
-x & b_1 & 0 & b_3 & \ldots & b_n
\end{pmatrix}$, where  $D = 
\begin{pmatrix}
y & a_3 & \ldots & a_n \\ 
-x & b_3 & \ldots & b_n
\end{pmatrix}$ is 1-generic.

\item $M' =
\begin{pmatrix}
y & 0 & a_2 & a_3 & a_4 & \ldots & a_n \\ 
-x & b_1 & 0  & \lambda a_3 & b_4 & \ldots & b_n
\end{pmatrix}$, where $\lambda$ is a scalar and there are no additional restrictions.

\end{enumerate}
\end{prop}

The next result shows that only the last three cases of the above proposition are relevant to our study of ideals defining Koszul algebras, since such an ideal must have at least two independent linear syzygies by Corollary \ref{nonACI:Koszul:algebras:have:2:linear:syzygies}.

\begin{cor} \label{multiplicity:1:with:no:linear:syzygies}
Suppose $I \subseteq S$ is a height two ideal of quadrics represented by minors by a matrix $M$ of linear forms as in \textnormal{(1)} or \textnormal{(2)} of the above proposition.  Then $I$ has no linear syzygies.
\end{cor}

\begin{proof}
Assume first that $M$ is 1-generic.  In that case, we must have $\hgt(y, a_1, \dots, a_n) = n+1$ or else $M$ would have a generalized zero. By \cite[5.1]{projective:dimension:of:height:2:ideals:of:quadrics}, we know that $(I : y) = I_2(M)$ and $(I, y) = (y, a_1x, \dots, a_nx)$ so that the rings $S/(I : y)$ and $S/(I, y)$ have the following Betti tables respectively:
\begin{center}
\begin{tabular}{c|cccccccc}
  & 0 & 1 & 2 & $\cdots$ & $n$  \\ 
\hline 
0 & 1 & -- & --  & -- & --  \\ 
1 & -- & $\binom{n+1}{2}$ & $2\binom{n+1}{3}$  & $\cdots$ & $n\binom{n+1}{n+1}$ 
\end{tabular}
\hspace{1 cm}
\begin{tabular}{c|cccccccc}
  & 0 & 1 & 2 & 3 & $\cdots$ & $n+1$ \\ 
\hline 
0 & 1 & 1 & --  & -- & -- & -- \\ 
1 & -- & $n$ & $\binom{n+1}{2}$  & $\binom{n+1}{3}$ & $\cdots$ & 1 
\end{tabular}
\vspace{1 em}
\end{center}
The former Betti table follows from the Eagon-Northcott resolution since $M$ is 1-generic \cite[6.2, 6.4]{geometry:of:syzygies}, and the latter follows by taking a mapping cone of multiplication by $y$ on the resolution of $S/(a_1x, \dots, a_nx)$.  The exact sequence $0 \to S/(I : y)(-1) \stackrel{y}{\to} S/I \to S/(I, y) \to 0$ induces an exact sequence
\[ 0 \to \Tor_2^S(S/I, \kk)_3 \to \Tor_2^S(S/(I, y), \kk)_3 \to \Tor_1^S(S/(I : y), \kk)_2 \to 0 \]
from which it follows that $\beta_{2,3}^S(S/I) = 0$ as wanted.

When $M$ is as in case (2) of the preceding proposition, the corollary follows by a similar argument with the following differences.  In this case, we have $\hgt(y, a_2, \dots, a_n) = n$ and $(I : y) = (b_1, I_2(D))$ so that $S/(I : y)$  and $S/(I, y)$ have the following Betti tables respectively
\begin{center}
\begin{tabular}{c|cccccccc}
  & 0 & 1 & 2 & $\cdots$ & $n-1$ & $n$  \\ 
\hline 
0 & 1 & 1 & --  & -- & -- & --  \\ 
1 & -- & $\binom{n}{2}$ & $\beta_2$  & $\cdots$ & $\beta_{n-1}$ & $(n-1)\binom{n}{n}$ 
\end{tabular}
\hspace{1 cm}
\begin{tabular}{c|cccccccc}
  & 0 & 1 & 2 & 3 & $\cdots$ & $n$ \\ 
\hline 
0 & 1 & 1 & --  & -- & -- & -- \\ 
1 & -- & $n-1$ & $\binom{n}{2}$  & $\binom{n}{3}$ & $\cdots$ & 1 
\end{tabular}
\vspace{1 em}
\end{center}
where $\beta_i = (i-1)\binom{n+1}{i+1} + \binom{n}{i+1}$ for $1 < i < n$.  Since $I_2(D)$ is a prime ideal generated by quadrics, it follows that $b_1$ is nonzerodivisor modulo $I_2(D)$, and therefore, the Betti table of $S/(I : y)$ is obtained by by taking a mapping cone of multiplication by $b_1$ on the resolution of $S/I_2(D)$.
\end{proof}

\begin{lemma}\label{twoht1}
Suppose that $I = (q_1, q_2, q_3, q_4) \subseteq S$ is a height two ideal minimally generated by four quadrics with $\hgt(q_1, q_2) = \hgt (q_3, q_4) = 1$ and $e(R) = 1$.  Then $\pd_S R \leq 4$.
\end{lemma}

\begin{proof}
By assumption, $I = (a_1x, a_2x, b_3y, b_4y)$ for some linear forms $a_i$, $b_i$, $x$, and $y$ such that $\hgt(x, y) = 2$.  If $\hgt(x, b_3, b_4) = 2$, then without loss of generality after a suitable change of generators, we may assume that $b_3 = x$.  But this implies that $e(R) = 2$ by Corollary \ref{multiplicity:2:case:Betti:tables}, contrary to our assumptions.  Hence, we must have $\hgt(x, b_3, b_4) = 3$.  In that case, $(I, x) = (x, b_3y, b_4y)$ where $b_3y$ and $b_4y$ are nonzero quadrics modulo $x$.  Consequently, over the polynomial ring $\bar{S} = S/(x)$ we have $\pd_{\bar{S}} S/(I, x) \leq 2$ so that $\pd_S S/(I, x) \leq 3$ by \cite[4.3.3]{Weibel}.  On the other hand, we have $(I : x) = (a_1, a_2) + ((b_3y, b_4y) : x) = (a_1, a_2, b_3y, b_4y)$ since $\hgt(x, b_3, b_4) = 3$ and $\hgt(x, y) = 2$.  Again, over the polynomial ring $\bar{S} = S/(a_1, a_2)$, the image of $(I : x)$ is generated by two quadrics.  Thus, we have $\pd_{\bar{S}} S/(I : x) \leq 2$ so that $\pd_S S/{(I : x)} \leq 4$.  The short exact sequence $0 \to S/(I : x)(-1) \stackrel{x}{\to} R \to S/(I, x) \to 0$ then implies that $\pd_S R \leq \max\{\pd_S S/(I : x), \pd_S S/(I, x)\} \leq 4$ as wanted. 
\end{proof}

\begin{lemma}\label{ht1}
Suppose that $I = (q_1, q_2, q_3, q_4) \subseteq S$ is a height two ideal minimally generated by four quadrics having at least two linear syzygies and that $\hgt(q_i, q_j)=1$ for some $i \neq j$.  Then $\pd_S R \leq 4$.
\end{lemma}

\begin{proof}
We may assume that $\hgt(q_1,q_2)=1$ and $e(R) = 1$. Let $(x,y)$ be the only minimal prime of height 2 of $I$.  By assumption, there exist linear forms $c, d_1, d_2$, where $d_1, d_2$ are linearly independent, such that $q_i=cd_i$ for $i=1,2$. Then either $c \in (x, y)$, or $c \notin (x, y)$ and $(d_1, d_2)\subseteq (x, y)$, which implies $(d_1,d_2)=(x, y)$. 

First assume $c\in (x,y)$. Then we can write $q_1 = a_1x$, $q_2 = a_2x$, and $q_i = a_ix + b_iy$ for $i = 3, 4$ for suitable linear forms $x, y, a_i, b_i$ with $\hgt(x, y) = 2$. Now, $(I, x)=(x, q_3, q_4)$ so that $\pd_S S/(I, x) \leq 3$. To prove the statement, it suffices to show that $\pd_S S/(I : x) \leq 4$. 

Observe that $(I : x) = (a_1, a_2)+((q_3, q_4) : x)$. By the preceding lemma, we may further assume that $\hgt (q_3, q_4) = 2$.  Then
\[ ((q_3, q_4) : x) = ((q_3, q_4) : (I, x)) = ((q_3, q_4): (I, x)^{\unmixed}) = ((q_3, q_4) : (x, y)) \]
where the leftmost equality follows because $(q_3,q_4,x)=(I,x)$, and the rightmost equality because $x$ lies in the only minimal prime of $I$, thus $(I,x)^{\unmixed}=(x,y)$. It follows that $((q_3,q_4) : x) = ((q_3,q_4) : (x,y))$ is linked via $(q_3,q_4)$ to the complete intersection $(x,y)$.  Thus, by Theorem \ref{AKM}, we have $((q_3,q_4): x) = I_2(M)$, where $M$ is the $3 \times 2$ matrix of linear forms
\[
M = 
\begin{pmatrix}
-a_4 & -b_4 \\
a_3 & b_3 \\
-y & x
\end{pmatrix}
.
\]
Then $(I : x) = (a_1, a_2, I_2(M))$.  However, since $I$ has at least two independent linear syzygies, there is a linear syzygy $\ell = (\ell_1, \ell_2, \ell_3, \ell_4)$ which is not a scalar multiple of $a =(a_2, -a_1, 0, 0)$.  If $\ell_1a_1 + \ell_2a_2 = 0$, then $(\ell_1, \ell_2, 0, 0)$ is a scalar multiple of $a$ so that we can replace $\ell$ with $(0, 0, \ell_3, \ell_4)$, but this contradicts that $(q_3, q_4)$ is a complete intersection of quadrics.  Hence, we see that $\ell_1a_1 + \ell_2a_2 \in ((q_3, q_4) : x) = I_2(M)$ is a minimal generator, and this implies that $(I : x)/(a_1, a_2) = (a_1, a_2, I_2(M))/(a_1, a_2)$ is generated by at most two independent quadrics so that $\pd_S S/(I : x) \leq 4$, by an argument similar to the one near the end of the proof of Lemma \ref{twoht1}.

Next, assume $(d_1,d_2)=(x,y)$ and $c \notin (x, y)=I^{\unmixed}$.  Then $I=(cx, cy, q_3, q_4)$, and $\hgt (I, c) \geq 3$.  It follows that $(I, c)$ is a complete intersection of height 3 and $\pd_S S/(I,c) =3$. Moreover, since $c\notin (x,y)=I^{\unmixed}$, we clearly have $(x,y)\subseteq {(I:c)} \subseteq ((x, y): c) = (x, y)$ so that $(I:c) = (x,y)$ and $\pd_S S/(I:c) =2$.  From the short exact sequence $0 \to S/(I : c) \to S/I \to S/(I,c) \to 0$, we deduce that $\pd_S S/I \leq 3$, concluding the proof.
\end{proof}

\begin{lemma}
Let $I=(q_1,q_2,q_3,q_4) \subseteq S$ be an ideal of height two generated by four linearly independent quadrics with $e(R) = 1$, and set $J = (q_1, q_2, q_3)$.  If $(J : q_4)$ contains two independent linear forms, then $\pd_S R \leq 4$. 
\end{lemma}

\begin{proof}
Let $\ell$ and $\ell'$ be the independent linear forms contained in $(J : q_4)$.  We note that $\hgt J = \hgt I = 2$.  If not, then $q_i = a_ix$ for some linear forms $a_i$ and $x$ for $i = 1, 2, 3$, and Corollary \ref{multiplicity:2:case:Betti:tables} implies that $e(R) = 2$, contrary to our assumption that $e(R)=1$.  Since $\kk$ is infinite and $\hgt J = 2$, we may further assume that $\hgt(q_1, q_2) = 2$.  

We first observe that the statement holds when $e(S/J) \geq 2$.  Consider the short exact sequence 
\begin{equation} \label{colon:exact:sequence} 
0 \longto S/(J : q_4)(-2) \stackrel{q_4}{\longto} S/J \longto R \longto 0 .
\end{equation}
Since the multiplicity symbol is additive along exact sequences (see \cite[V.A.2]{Serre:local:algebra}), we must have $\hgt (J : q_4) = 2$.  As $(\ell, \ell')$ is a height two prime ideal in $(J : q_4)$, it follows that $(J : q_4) = (\ell, \ell')$ is a complete intersection. Since $J$ is a height two ideal generated by three quadrics, it follows from \cite{projective:dimension:of:height:2:ideals:of:quadrics} that $\pd_S S/J \leq 4$ so that $\pd_S R \leq \max\{\pd_S S/(J: q_4) + 1, \pd_S S/J\} \leq 4$. 

We may therefore assume that $e(S/J) = 1$.  In that case, we must have $J^{\unmixed} = (x, y)$, where $x$ and $y$ are linear forms generating the unique prime ideal of height 2 containing $I$.  Let $C = (q_1, q_2)$, which we have assumed is a complete intersection. We then have 
\[
(C : q_3) = (C : J) = (C : J^{\unmixed}) = (C : (x,y)).
\]
As in the proof of the previous lemma, $(C : q_3)$ is a Northcott ideal of height two directly linked to $(x, y)$; in particular, it is Cohen-Macaulay so that ${\pd_S S/(C : q_3)} = 2$. It then follows from the short exact sequence
\[
0 \longto S/(C : q_3)(-2) \stackrel{q_3}{\longto} S/C \longto S/J \longto 0
\] 
that $\pd_S S/J \leq \max\{\pd_S S/(C: q_3) + 1, \pd_S S/C\} = 3$. In particular, we have $\hgt Q \leq 3$ for every $Q \in \ass(S/J)$ by the Auslander-Buchsbaum formula after localizing at $Q$. Since $q_4 \in (x, y)=J^{\unmixed}$ and $q_4 \notin J$, it follows that $(J : q_4)$ is a proper ideal of height at least 3. Since $\Ass(S/(J:q_4))\subseteq \Ass(S/J)$ and every associated prime of $S/J$ has height at most three, it follows that $\hgt (J : q_4) = 3$. Since $(J : q_4)$ contains the height two linear prime $(\ell, \ell')$, then $(J : q_4) = (\ell, \ell')+ fH$ for some form of positive degree $f \notin (\ell, \ell')$ and an ideal $H \subseteq S$ such that  either $H = S$ or $(\ell, \ell') + H$ is a proper ideal of height at least four. 

We claim that $H = S$.  If not, then $(\ell, \ell') + H = ((J : q_4) : f)$ is a proper ideal of height at least four.  However, since $\ass(S/((J : q_4) : f)) \subseteq \ass(S/(J : q_4)) \subseteq \ass(S/J)$, we see that $(\ell, \ell') + H$ has height at most 3, which is a contradiction.

Therefore, $(J : q_4) = (\ell, \ell', f)$ is a complete intersection of height three. As $\pd_S S/J \leq 3$ and $\pd_S S/(J : q_4) = 3$, \eqref{colon:exact:sequence} yields $\pd_S R \leq \max\{{\pd_S S/(J : q_4)} + 1, \pd_S S/J\} \leq 4$.
\end{proof}

We can now prove Theorem \ref{projective:dimension:bound}.

\begin{proof}[Proof of Theorem \ref{projective:dimension:bound}]
By Theorem \ref{multiplicity:2:case} and Corollary \ref{multiplicity:2:case:Betti:tables}, we may assume that $e(R) = 1$.  By Corollary \ref{multiplicity:1:with:no:linear:syzygies}, we know that $I$ can be represented by minors by a matrix as in cases (3)--(5) of Proposition \ref{representation:by:minors}.  We may further assume that $\hgt (q_i, q_j) = 2$ for all $i \neq j$ by Lemma \ref{ht1} so that  we are in case (4) or (5).  In particular, we may assume that $q_1 = b_1y$ and $q_2 = a_2x$, where $x$ and $y$ are the linear forms generating the unique minimal prime of $I$ of height 2.  Let $\ell = (\ell_1,\ell_2,\ell_3,\ell_4)$ and $h = (h_1,h_2,h_3,h_4)$ in $S(-2)^4$ be linearly independent syzygies of $I$.  If $\ell_4$ and $h_4$ are independent linear forms, then we are done by the preceding lemma.  So we may assume that $\ell_4$ and $h_4$ are linearly dependent, and therefore, after possibly switching $\ell$ and $h$ and subtracting a multiple of $h$ from $\ell$, we may further assume that $\ell_4 = 0$.  

In this case, we claim there exists a suitable choice of generators $I = (q'_1, q'_2, q'_3, q_4)$ such that two of the generators generate a height one ideal so that the conclusion follows from Lemma \ref{ht1}. Consider the complete intersection $C = (q_2, q_3)$, and write $q_i=a_ix+b_iy$ for $i=2,3$. If $\hgt(y, a_2, a_3) = 2$, then we have $\alpha y + \beta a_2 + \gamma a_3 = 0$ for some $\alpha, \beta, \gamma \in \kk$ not all zero.  If $\gamma = 0$, then $a_2 \in (y)$, and if $\gamma \neq 0$, then $a_3 \in (a_2, y)$.  Hence, in either case, after a suitable change of generators, we see that $y$ divides two generators of $I$, and we are done.  So we may further assume that $\hgt (y, a_2, a_3) = 3$.  Then $(C, y)^{\unmixed} = (y, a_2x, a_3x)^{\unmixed} = (x, y)$.  Indeed, it is clear that $(x, y)$ is the unique height two prime ideal containing $(C, y)$, and $(C, y)S_{(x, y)} = (x, y)S_{(x, y)}$ since $(y, a_2, a_3) \nsubseteq (x, y)$ implies either $a_2$ or $a_3$ is not in $(x, y)$.  And so, it follows that $(C : y) = (C : (C, y)) = (C : (C, y)^{\unmixed}) = (C : (x, y))$.  As $\ell_1b_1 \in (C: y)$, we have $\ell_1b_1x, \ell_1b_1y \in C$.  We note that $\ell_1 \neq 0$ since otherwise we would have a linear syzygy $\ell_2q_2 + \ell_3q_3 = 0$, contradicting that $C$ is a complete intersection.  If $b_1$ is a nonzerodivisor modulo $C$, then $\ell_1x, \ell_1y \in C$ are independent quadrics so that $C = (\ell_1x, \ell_1y)$, contradicting that $C$ has height two.  Hence, $b_1$ is a zerodivisor modulo $C$ so that $(C, b_1)$ is a height two ideal.  Then $C' = (C, b_1)^{\unmixed}$ is an unmixed height two ideal containing a linear form, and therefore, it is a complete intersection.  Consequently, we have that $(C : b_1) = (C : C')$ is linked in one step to a complete intersection.  Thus, $(C : b_1) = (C, f)$ is a Cohen-Macaulay height two ideal for some homogeneous form $f$ of positive degree by Theorem \ref{AKM}, and this ideal contains the independent quadrics $\ell_1x$ and $\ell_1y$.  This implies that there is a nonzero linear form $z \in (x, y)$ such that $\ell_1z \in C$.  Indeed, we can write $\ell_1x = \alpha_1f + \alpha_2q_2 + \alpha_3q_3$ and $\ell_1y = \beta_1f + \beta_2q_2 + \beta_3q_3$ for some $\alpha_i, \beta_i \in \kk$.  Note that $\alpha_1$ and $\beta_1$ are not both zero, otherwise we would have $C = (\ell_1x, \ell_2 y)$ contradicting that $\hgt C = 2$.  Hence, $\ell_1z \in C$ for $z = \beta_1x - \alpha_1y$.

Without loss of generality, we can then replace $q_2$ with $q'_2 = \ell_1z$ as a generator for $C$ and still have a linear syzygy of the form $\ell_1(b_1y) + \ell_2q'_2 + \ell_3q_3 = 0$. In that case, as $(q'_2, q_3) = C \subseteq (\ell_1, q_3)$, we see that $\hgt (\ell_1, q_3) = 2$ so that the above syzygy yields $\ell_3 \in (\ell_1: q_3) = (\ell_1)$.  If $\ell_3 = 0$, then $\hgt(q_1, q'_2) = 1$, and if $\ell_3 \neq 0$, then $q_3 = \lambda(q_1 + \ell_2z)$ for some $\lambda \in \kk$ so that replacing $q_3$ with $q'_3 = \ell_2z$ as a generator for $I$ yields $\hgt(q'_2, q'_3) = 1$.
\end{proof}

Combining the results of the preceding sections, we can determine all Betti tables of Koszul algebras defined by height two ideals of four quadrics; they are precisely the Betti tables we expect from Table \ref{height:2:4-generated:edge:ideals}.

\begin{thm} \label{Koszul:4:quadric:height:2:Betti:tables}
Let $R = S/I$ be a Koszul algebra defined by four quadrics with $\hgt I = 2$.  Then the Betti table of $R$ over $S$ is one of the following:
\vspace{1 ex}
\begin{center}
\begin{minipage}{\textwidth}
\begin{multicols}{2}
\begin{enumerate}[label = \textnormal{(\roman*)}]
\item 
\textnormal{
\begin{tabular}{c|cccccccc}
  & 0 & 1 & 2 & 3  \\ 
\hline 
0 & 1 & -- & --  & --\\ 
1 & -- & 4 & 4  & 1 
\end{tabular}
}

\vspace{1 ex}

\item
\textnormal{
\begin{tabular}{c|cccccccc}
  & 0 & 1 & 2 & 3 & 4 \\ 
\hline 
0 & 1 & -- & --  & -- & -- \\ 
1 & -- & 4 & 3  & 1  & -- \\
2 & -- & -- & 3 & 3 & 1
\end{tabular}
}

\item
\textnormal{
\begin{tabular}{c|cccccccc}
  & 0 & 1 & 2 & 3  \\ 
\hline 
0 & 1 & -- & --  & -- \\ 
1 & -- & 4 & 3  & -- \\
2 & -- & -- & 1 & 1 
\end{tabular}
}

\vspace{1 ex}

\item
\textnormal{
\begin{tabular}{c|cccccccc}
  & 0 & 1 & 2 & 3 & 4 \\ 
\hline 
0 & 1 & -- & --  & -- & -- \\ 
1 & -- & 4 & 2  & --  & -- \\
2 & -- & -- & 4 & 4 & 1
\end{tabular}
}
\end{enumerate}
\end{multicols}
\end{minipage}
\end{center}
In particular, we have $\beta_i^S(R) \leq \binom{4}{i}$ for all $i$.
\end{thm}

\begin{proof}
By Theorem \ref{multiplicity:at:most:2}, we know that $e(R) \leq 2$, and if $e(R) = 2$, then the Betti table of $R$ is either (i) or (ii) by Theorem \ref{multiplicity:2:case} and Corollary \ref{multiplicity:2:case:Betti:tables}.  It remains to show that the Betti table of $R$ is either (iii) or (iv) if $e(R) = 1$.  In that case, we know that $2 = \hgt I \leq \pd_S R \leq 4$ by Theorem \ref{projective:dimension:bound}.  If $\pd_S R = \hgt I = 2$, then $R$ is Cohen-Macaulay so that $I = I_3(M)$ for a $4 \times 3$ matrix $M$ of homogeneous forms of positive degree \cite[1.4.17]{Bruns:Herzog}, which is clearly impossible for an ideal generated by quadrics. 

Suppose that $\pd_S R = 3$.  Since $R$ is not a complete intersection, it follows from Corollary \ref{regularity:bounded:by:projective:dimension} that $\reg R \leq 2$ so that the Betti table of $R$ has the form:
\begin{center}
\begin{tabular}{c|cccccccc}
  & 0 & 1 & 2 & 3  \\ 
\hline 
0 & 1 & -- & --  & -- \\ 
1 & -- & 4 & $a$  & $c$ \\
2 & -- & -- & $b$ & $d$ 
\end{tabular}
\end{center}
We show $c = 0$. If not, then after killing a maximal regular sequence of linear forms, we may assume that $S = \kk[a_1, a_2, a_3]$.  Then by Proposition \ref{high:linear:syzygy:implies:nonzero:socle}, there is a linear form $x \in S$ such that $a_ix \in I$ for all $i$ so that $I = (a_1x, a_2x, a_3x, q)$, but we have already seen in Corollary \ref{multiplicity:2:case:Betti:tables} that an ideal of this form has Betti table (i) or (ii), which is a contradiction.  Thus, $c=0$.

The Hilbert series of $R$ then has the form $H_R(t) = Q(t)/(1-t)^{\dim S}$ where $Q(t) = 1 -4t^2 +at^3 +bt^4 - dt^5$.  Since $\hgt I = 2$ and $e(R) = 1$, we know that $Q(1) = Q'(1) = 0$ and $\frac{Q''(1)}{2} = 1$ \cite[\S 4.1]{Bruns:Herzog}\footnote{However, the reader should consult the errata for Corollary 4.1.14.}, which yields the following system of equations
\begin{align*}
a + b -d &= 3 \\
3a  + 4b - 5d &= 8 \\
3a + 6b - 10d &= 5 
\end{align*}
whose only solution is easily checked to be $a = 3$, $b = 1$, and $d = 1$.  

Suppose now that $\pd_S R = 4$.  As above, we see that $\reg R \leq 3$ by Corollary \ref{regularity:bounded:by:projective:dimension} and that $\beta_{2,5}^S(R) = \beta_{3,6}^S(R) = \beta_{4,5}^S(R) = 0$ by Lemma \ref{Koszul:algebras:have:subdiagonal:Betti:table}, Proposition \ref{Koszul:Betti:table:constraints}, and Corollary \ref{projective:dimension:level:linear:syzygy} as $\hgt I = 2$ so that the Betti table of $R$ has the form:
\begin{center}
\begin{tabular}{c|cccccccc}
  & 0 & 1 & 2 & 3 & 4 \\ 
\hline 
0 & 1 & -- & --  & -- & -- \\ 
1 & -- & 4 & $a$  & $c$ & -- \\
2 & -- & -- & $b$ & $d$ & $e$ \\
3 & -- & -- & -- & -- & $f$
\end{tabular}
\end{center}
In this case, we have $Q(t) = 1 -4t^2 +at^3 +(b-c)t^4 - dt^5 + et^6 + ft^7$ so that the equalities $Q(1) = Q'(1) = 0$ and $\frac{Q''(1)}{2} = 1$ translate to the following system of equations:
\begin{align*}
a + b - c -d +e + f &= 3 \\
3a  + 4b -4c - 5d +6e + 7f &= 8 \\
3a + 6b -6c - 10d +15e + 21f &= 5 
\end{align*}
This system is easily checked to reduce to the equivalent system of equations:
\begin{align*}
a +e + 3f &= 3 \\
b -c -3e - 8f &= 1 \\
d -3e - 6f &= 1 
\end{align*}
Since $a, e, f$ are non-negative integers and $a \geq 2$ by Corollary \ref{nonACI:Koszul:algebras:have:2:linear:syzygies}, the first of the above equations forces $f = 0$ and $e \leq 1$, which must be an equality since $\pd_S R = 4$.  Hence, $a = 2$ so that $c = 0$ by part (b) of \cite[4.2]{Koszul:algebras:defined:by:3:quadrics}. The remaining values are easily computed, and we see that the Betti table of $R$ must be (iv).
\end{proof}

\begin{cor} \label{Betti:number:4:quadrics} 
Question \ref{Betti:number:bound:for:Koszul:algebras} has an affirmative answer for all Koszul algebras defined by $g = 4$ quadrics.
\end{cor}

\begin{proof}
It is immediate that the Betti number bound for a Koszul algebra $R = S/I$ holds when $I$ is a complete intersection ($\hgt I = 4$), when $\hgt I = 1$ by Proposition \ref{Koszul:Betti:table:constraints}, and when $I$ is an almost complete intersection ($\hgt I = 3$) by \cite{Koszul:ACI's}.  Combining this with the preceding theorem on the possible Betti tables of Koszul algebras when $\hgt I = 2$ completes the proof.
\end{proof}

\end{spacing}

\end{document}